\documentclass[reqno]{amsart}
\usepackage{amsfonts}
\usepackage{amssymb}
\usepackage{enumerate} 
\usepackage[lite]{amsrefs}

\title[Constructive multi-typed theory and second order arithmetic]{Comparison of constructive multi-typed theory with subsystems of second order arithmetic}

\date{}



\author{Farida Kachapova}
\address{School of Computer and Mathematical Sciences\\
Auckland University of Technology\\
Auckland, New Zealand}

\email{farida.kachapova@aut.ac.nz}

\newpage

\newtheorem{theorem}{Theorem}[section]

\newtheorem{lemma}[theorem]{Lemma}

\newtheorem{corollary}[theorem]{Corollary}

\theoremstyle{definition}

\begin{document}

\begin{abstract}
This paper describes an axiomatic theory $BT$ for constructive mathematics. $BT$ has a predicative comprehension axiom for a countable number of set types and usual combinatorial operations. $BT$ has intuitionistic logic, is consistent with classical logic and has such constructive features as consistency with formal Church thesis, and existence and disjunction properties. $BT$ is mutually interpretable with a so called theory of arithmetical truth $PATr$ and with a second-order arithmetic $SA$ that contains infinitely many sorts of sets of natural numbers. We compare $BT$ with some standard second-order arithmetics and investigate the proof-theoretical strengths of fragments of $BT$, $PATr$ and $SA$.
\end{abstract}

\maketitle

\section{Introduction}

Beeson \cites{bees78a, bees78, bees85} and Feferman \cites{fefe75, fefe79} introduced axiomatic theories containing operations and sets. These theories are intended for developing constructive mathematics in Bishop's style \cites{bish12,bish11}. The theories have intuitionistic logic and are consistent with classical logic. Kashapova \cite{kash84} generalised the Beeson's theory $BEM+(CA)$ \cite{bees78} to a language with infinitely many types of sets. The resulting axiomatic theory was studied further in \cite{kach14a} and \cite{kach14}. In \cite{kach14a} we constructed a realizability and a set-theoretical model for $BT$, and proved existence property of $BT$. In \cite{kach14} we constructed an interpretation of $BT$ in a so called theory of arithmetical truth $PATr$ obtained from Peano arithmetic by adding infinitely many truth predicates. 

In this paper we study other metamathematical properties of the theory $BT$. One of them is the disjunction property. We show that $BT$ is mutually interpretable with $PATr$ and with a second-order arithmetic $SA$ containing infititely many sorts for sets of natural numbers. We also show that $BT$ is interpretable in the second-order arithmetic with $\triangle^1_1$ comprehension axiom. We prove that each fragment $BT_s$ with types $\leqslant s$ is weaker than next fragment $BT_{s+1}$. In particular, $BT$ is stronger that the Beeson's theory $BEM+(CA)$.

In section \ref{section:BT} we give a detailed definition of the theory $BT$ and describe its constructive properties: existence and disjunction properties, and consistency with formal Church thesis. 

In section \ref{section:SA} we define the multi-sorted arithmetic $SA$ and interpret it in the second-order arithmetic with $\triangle^1_1$ comprehension axiom.

In section \ref{sec:main} we define the theory of arithmetical truth $PATr$. In section \ref{section:interpretability} we prove that the theories $BT$, $BT$ with classical logic, $SA$ and $PATr$ are interpretable in one another, and so are their corresponding fragments. 

In section \ref{section:comparison} we show that each fragment $BT_{s+1}$ is proof-theoretically stronger than the previous fragment $BT_{s}$ $(s\geqslant 0)$. The same is proven for corresponding fragments of the theories $PATr$ and $SA$. 

In the rest of the introduction we explain some notations and terminology. 

All theories considered in this paper are first-order axiomatic theories (a well-known definition of a first-order axiomatic theory can be found, for example, in \cite{mend09}).

The symbol $\leftrightharpoons$ means ``equals by definition". The symbol $\diamond$ denotes a logical connective $\wedge,\vee$ or $\supset$, and the symbol $Q$ denotes a quantifier $\forall$ or $\exists$. In each of our axiomatic theories we have the logical constant $\bot$ for falsity and we regard $\neg \varphi$ as an abbreviation for $\varphi\supset\bot$. The complexity of a formula $ \varphi $ is the number of occurrences of logical symbols (the main three connectives and quantifiers) in $ \varphi $. For any formula $\varphi$ we denote $\overline{\overline{\varphi}}$ the closure of $\varphi$, that is, the formula $\varphi$ with universal quantifiers over all its parameters. We denote $\tau\left[x_1,\ldots,x_n/t_1,\ldots,t_n \right]$ the result of proper substitution of terms $t_1,\ldots,t_n$ for variables $x_1,\ldots,x_n$ in an expression $\tau$. The complexity of a term $t$ is the number of occurrences of functional symbols in $t$.

We fix a one-to-one coding of all finite sequences of natural numbers such that 0 is the code for the empty sequence. In a theory containing first-order arithmetic we use the notations:
\begin{list}{•}{•}
\item $(n_1,\ldots,n_k)$ as the code for sequence $n_1,\ldots,n_k$;
\item $(n)_i$ for the $i$th element of the sequence with code $n$;
\item $lh(n)$ for the length of the sequence with code $n$. 
\end{list}

We fix a numbering of partial recursive functions and denote $\{e\}(n)$ the value at $n$ of the partial recursive function with number $e$ if this value is defined.

For a natural number $n$ we denote $\overline{n}$ the formal arithmetical term for $n$, that is $\overline{n}=\underbrace{1+1\ldots +1}_{n\;times}$. 
\medskip

We assume that for any axiomatic theory $K$ some G\"{o}del numbering of its expressions is fixed. For an expression $q$ we denote $\llcorner q \lrcorner$ the G\"{o}del number of $q$ in this numbering; $t_m$ and $\varphi_m$ denote the term and formula with G\"{o}del number $m$, respectively. 

The notation $K\vdash \varphi$ means that formula $\varphi$ is derivable in theory $K$. The theory $K$ is consistent if it is not true that $K\vdash \bot$. 
$Proof_K(m,n)$ denotes the arithmetical formula stating that $n$ is the G\"{o}del number of a formal proof in the theory $K$ for formula $\varphi_m$. The formula $Pv_K(m)\leftrightharpoons\exists n Proof_K(m,n)$ means that $\varphi_m$ is derivable in the theory $K$. The formula $Con_K\leftrightharpoons \neg Pv_K(\llcorner\bot\lrcorner)$ means that the theory $K$ is consistent.

In this paper we consider axiomatic theories where variables have superscripts for types or sorts. A superscript for a variable is usually omitted when the variable is used for the second time or more in a formula or in a proof (so its type or sort is obvious). In metamathematical proofs we use classical logic.

\section{Axiomatic theory $BT$} 
\label{section:BT}
\subsection{Definition of theory $BT$}

Theory $BT$ was first introduced in  \cite{kash84} as a generalisation of Beeson's theory $BEM+(CA)$ \cite{bees78}. The language of the theory $BT$ has the following variables:
\begin{list}{}{}
\item $m,n,\ldots$ over natural numbers (variables of type $\omega$) and 
\item $X^k, Y^k, Z^k,\ldots$ of type $k$ $(k=0,1,2,\ldots)$.
\end{list}

We will identify the variables $X^0, Y^0, Z^0,\ldots$ of type 0 with variables $x,y,z,\ldots$, respectively, which we call \textit{operation variables}. We consider the type $\omega$ to be smaller than any other type. Variables of type 0 are interpreted as operations and variables of types $\geqslant 1$ are interpreted as sets. 

$BT$ has a numerical constant 0 and the following operation constants:
\begin{list}{•}{•}
\item combinatorial constants $\underline{k}, \underline{s}, \underline{d},\underline{p},\underline{p}_1,\underline{p}_2$
\medskip
\item and comprehension constants $\underline{c}_n$ $(n\geqslant 0)$, which are used for constructing sets.
\end{list}

There are no functional symbols in $BT$.

Predicate  symbols: 

$Ap(f,x,y),\; x=_{0\omega}m,\;x=_{0k}Y^k ,\; X^k\in_{k}Y^{k+1} (k\geqslant 0)$.
\medskip

$Ap(f,x,y)$ means that $y$ is the result of application of operation $f$ to $x$.

Atomic formulas are obtained from predicate symbols by substituting constants and variables of corresponding types. Formulas are constructed from atomic formulas and $\bot$ using logical connectives and quantifiers.

The language of $BT$ is defined. A formula of $BT$ is said to be $n$-\textit{elementary} if it contains only types $\leqslant n$, no quantifiers over variables of type $n$ and no predicate symbol $=_{0n}$.

External terms are defined recursively as follows.
\begin{enumerate}
\item Every constant and variable is an external term.
\item If $t$ and $\tau$ are external terms, then $t\tau$ is an external term.
\end{enumerate}

$t\tau$ is interpreted as the result of application of operation $t$ to $\tau$. External terms are generally not part of the language $BT$. The notation $t_1t_2t_3\ldots t_n$ means $(\ldots((t_1t_2)t_3\ldots )t_n$. 

We consider each operation to have one argument. A function $f$ with $n$ arguments can be written as an operation that is applied $n$ times, i.e. instead of $f(x_1,x_2,\ldots, x_n)$ we use $(\ldots((f(x_1)x_2)\ldots )x_n)$.

For an external term $t$ we define a relation $t\simeq x$ by induction on the construction of $t$ as follows.

\begin{enumerate}
\item If $t$ is a constant or a variable of type $s$, then $t\simeq x\leftrightharpoons x=_{0s}t$.
\item If $t$ is $t_1t_2$, then $t\simeq x\leftrightharpoons \exists y,z(t_1\simeq y\wedge t_2\simeq z\wedge Ap(y,z,x))$.
\end{enumerate}

These are some more notations for external terms:
\begin{list}{•}{•}
\item $t\downarrow \leftrightharpoons \exists x(t\simeq x)$;
\medskip
\item $t\simeq\tau \leftrightharpoons \exists x(t\simeq x\wedge \tau\simeq x)$;
\medskip
\item $t\cong\tau \leftrightharpoons \forall x(t\simeq x\equiv \tau\simeq x)$;
\medskip
\item $\varphi(t) \leftrightharpoons \exists x(t\simeq x\wedge\varphi(x))$.
\end{list}

We fix G\"odel numbering of all expressions of the language $BT$. 

The theory $BT$ has the following axioms.

1. Intuitionistic predicate logic.
\medskip

2. Equality axioms

1) $x=_{00}x$.

2) $u=_{0k}X^k\wedge v=_{0k}X\wedge u=_{0n}Y^n\supset v=_{0n}Y (k\geqslant0,n\geqslant0)$.

3) $u=_{0\omega}m\wedge v=_{0\omega}m\supset u=_{00}v$.

4) $u=_{0\omega}m\wedge u=_{00}v \supset v=_{0\omega}m$.
\smallskip

5) $Ap(f,x,y)\wedge f=_{00}g\wedge x=_{00}u\wedge y=_{00}v\supset Ap(g,u,v)$.
\smallskip

6) $X^k\in_k Y^{k+1}\wedge X^k=_k U^k \wedge Y^{k+1}=_{k+1} Z^{k+1}\supset U\in_k Z (k\geqslant0)$.
\medskip

3. Combinatorial axioms

1) $Ap(f,x,y)\wedge Ap(f,x,z)\supset y=z$.

2) $\underline{k}xy \simeq x$.
\smallskip

3) $\underline{s}xy\downarrow $.\qquad
4) $\underline{s}xyz\cong xz(yz)$.
\smallskip

5) $\underline{p}xy\downarrow $.\qquad
6) $\neg(\underline{p}xy\simeq 0) $.
\smallskip

7) $\underline{p}_i x\downarrow(i=1,2) $.\qquad
8) $\underline{p}_i (\underline{p}x_1x_2)\simeq x_i(i=1,2) $.
\medskip

9) $\exists m(\underline{p}n0\simeq m) $.\qquad
10) $\exists Z^k(\underline{p}xY^k\simeq Z) $.
\smallskip

11) $n=m\supset\underline{d}xynm\simeq x$.\qquad
12) $n\neq m\supset\underline{d}xynm\simeq y$.
\smallskip

13) $\exists x(x=_{0\omega}n) $.\qquad
14) $\exists x(x=_{0k}Y^k) $.
\medskip

In $BT$ the successor of a natural number $n$ is given by $\underline{p}n0$. 
\medskip

Any natural number $m$ is represented by a term $\underbrace{\underline{p}(\ldots\underline{p}(\underline{p}}_m 0\underbrace{0)0\ldots)0}_m$, which we denote $\overline{m}$.

4. Induction over natural numbers

$\underline{\varphi[n/0], \quad\varphi\supset\exists m\left(\underline{p}n0\simeq m\wedge \varphi[n/m]\right) }$

$\qquad\qquad\varphi$
\smallskip
\\where $\varphi$ is any formula of $BT$.

Finite sequences are introduced in $BT$ using the pair operation:
\begin{list}{•}{•}
\item $\langle x_1\rangle \leftrightharpoons x_1$;
\item $\langle x_1,x_2,\ldots,x_{n+1}\rangle \leftrightharpoons \underline{p}(\langle x_1,x_2,\ldots,x_{n}\rangle)x_{n+1}$.
\end{list}

For brevity we will denote an external term $\tau(\langle t_1,t_2,\ldots,t_{n}\rangle)$ as \\$\tau( t_1,t_2,\ldots,t_{n})$.

5. Comprehension axiom
\[\exists U^{k+1}\left[\;\underline{c}_n(\widetilde{X})\simeq U\wedge\forall Z^k(Z\in_k U\equiv \varphi)
 \right],\]
where $n=\llcorner Z^k.\widetilde{X}.\varphi\lrcorner$, $\widetilde{X}$ is a finite list of variables of types $\leqslant (k+1)$, and $\varphi$ is a $(k+1)$-elementary formula with all its parameters in the list $Z^k, \widetilde{X}$.

This completes the definition of the theory $BT$.
We denote $BT_s$ the fragment of $BT$ containing only types not greater than $s$ $(s\geqslant0)$.

As usual in combinatorial logic, for any external term $t$ and variable $x$ we can construct an external term $\lambda x.t$ with the property:
\[BT\vdash \lambda x.t\downarrow \wedge (\lambda x.t)x\cong t.\]

Using $\lambda$-terms we can define in $BT_0$ recursion operator, $\mu$-operator and all primitive recursive functions. Thus, $BT_0$ contains the intuitionistic arithmetic $HA$.

\subsection{Constructive properties of $BT$}

In \cite{kach14a} we defined a realizability $fr\varphi$ for formulas of $BT$. In particular: \begin{equation}
fr(\psi\vee\chi) \leftrightharpoons
\exists k,u\left[f\simeq\langle k,u\rangle\wedge(k=0\supset ur\psi)\wedge(k\neq 0\supset ur\chi)
 \right].
\label{eq:disjunction1}
\end{equation}

The following lemma and theorem about the realizability were proven in \cite{kach14a}.

\begin{lemma}
$BT\vdash fr\varphi\supset\varphi$.
\label{lemma:realiz}
\end{lemma}

\begin{theorem} Soundness of the realizability.

If $BT\vdash\varphi$, then for some external term $t$, $BT\vdash t\downarrow\wedge \;tr\varphi$.
\label{theorem:realiz}
\end{theorem}

\begin{theorem}Existence property of $BT$.
\[\textit{If }BT\vdash \exists Y\varphi\textit{, then for some external term }t,
BT\vdash \exists Y(t\simeq Y)\wedge\varphi[Y/t].\]
Here $Y$ is a variable of any type.
\label{theorem:existence}
\end{theorem}
\begin{proof}
A proof using the realizability was given in \cite{kach14a}.
\end{proof}

Next we will show that $BT$ also has disjunction property. 

In \cite{kach14a} we constructed for each fragment $BT_s$ a set-theoretical model with domains, which are sets of external terms. In particular, the domain for numerical variables is $\mathcal{H}=\{\bar{m}\mid m\in \mathbb{N}
\}$. 
\medskip

The notation $\gamma\models^p\varphi$ means that in the model for $BT_{p-1}$ formula $\varphi$  holds under evaluation $\gamma$. 

\begin{theorem} Soundness of the model. 
\[BT_{p-1}\vdash \varphi\;\Rightarrow\;(\gamma\models^p \overline{\overline{\varphi}}),\] where $p\geqslant 1$ and $\gamma$ is the empty evaluation.
\label{theorem:set-model}
\end{theorem}

\begin{proof}
Proof was given in \cite{kach14a}.
\end{proof}

For external terms $t$ and $\tau$ the notation $t\succeq \tau$ means that $t$ can be reduced to $\tau$ using the properties of operation constants in axioms 3.2), 4), 8), 11) and 12). The following two lemmas were proven in \cite{kach14a}. 

\begin{lemma}
$t\succeq \tau\quad\Rightarrow \quad BT\vdash t\downarrow\;\supset (t\simeq \tau).$
\label{lemma:term-reduction1}
\end{lemma}

\begin{lemma}
$\gamma\models^p (t\simeq x)\quad\Leftrightarrow\quad t'\succeq \gamma(x),$
\medskip
\\where $t'$ is the external term  $t$ evaluated by $\gamma$.
\label{lemma:term-reduction2}
\end{lemma}

\begin{theorem}Disjunction property of $BT$. For closed formulas $\varphi$ and $\psi:$
\[\textit{if }BT\vdash \varphi\vee\psi\textit{, then }
BT\vdash \varphi\textit{ or }BT\vdash\psi.\]
\label{theorem:disjunction}
\end{theorem}
\begin{proof}
Suppose $BT\vdash \varphi\vee\psi$.  By Theorem \ref{theorem:realiz} for some external term $\tau$, $BT\vdash\tau\downarrow\wedge\tau r(\varphi\vee\psi)$. Denote $\tau'$ the closed term obtained from $\tau$ by substituting appropriate constants for all parameters of $\tau$ (that is, 0 for numerical parameters, $\underline{k}$ for operation parameters and $\underline{c}_{\llcorner Z^j.Z\neq Z\lrcorner}$ for parameters of type $j\geqslant 1$). Then $BT\vdash \tau'\downarrow\wedge\tau'r(\varphi\vee\psi)$. 
\medskip

Denote $t=\underline{p}_1\tau'$. By (\ref{eq:disjunction1}), $BT\vdash\exists k(t\simeq k)$ and by Lemma \ref{lemma:realiz}:
\begin{equation}
BT\vdash(t=0\supset\varphi)\wedge(t\neq 0\supset\psi).
\label{eq:disjunction2}
\end{equation}

Since any proof in $BT$ is finite, there is $p\geqslant 1$ such that 
\[BT_{p-1}\vdash\exists k(t\simeq k).\]

By Theorem \ref{theorem:set-model}, $\gamma\models^p \exists k(t\simeq k)$ for the empty evaluation $\gamma$ and by Lemma \ref{lemma:term-reduction2}, $(\exists r\in \mathcal{H})(t\succeq r)$. Therefore $t\succeq\overline{m}$ for some natural number $m$. By Lemma \ref{lemma:term-reduction1}, $BT\vdash t\simeq \overline{m}$. 
\smallskip

So if $m=0$, then $BT\vdash t=0$ and by (\ref{eq:disjunction2}) $BT\vdash\varphi$. 
\smallskip

If $m\neq0$, then $BT\vdash t\neq0$ and by (\ref{eq:disjunction2}) $BT\vdash\psi$. 
\end{proof}

In \cite{kash84} we showed that $BT$ is consistent with the following form of the formal Church thesis:
\[(\forall f\in \mathbb{N}^\mathbb{N})\exists e\forall n(fn\simeq \{e\}(n)).\]

The existence and disjunction properties of $BT$ as well as its consistency with the formal Church thesis are all evidence of the constructive nature of $BT$. 

\section{Multi-sorted second-order arithmetic $SA$}
\label{section:SA}

\subsection{Definition of theory $SA$}

The language of theory \textit{SA}  has the following variables:
\begin{list}{}{}
\item $n_1, n_2, \ldots,m,n, \ldots$ over natural numbers and 
\item $x_1^{(k)}, x_2^{(k)},\ldots, x^{(k)},y^{(k)},\ldots$ of sort $k$ over sets of natural numbers $(k=1,2,\ldots)$.
\end{list}

The language of $SA$ has two numerical constants 0 and 1, and functional symbols $\cdot$ and +. There are the following predicate  symbols: 

$=$ (equality of natural numbers) and $\in_k$ $(k=1,2,\ldots)$.
\medskip

Numerical terms are constructed from numerical variables and  constants using functional symbols. 
Atomic formulas are:

$t=\tau$; $t\in_k x^{(k)}$, where $t$ and $\tau$ are numerical terms.
\smallskip

Formulas are constructed from atomic formulas and $\bot$ using logical connectives and quantifiers.
The language of $SA$ is defined.

A formula $\varphi$ of $SA$ is called $k$\textit{-simple} if it has no quantifiers over set variables and it has no variables of sorts greater than $k$.

Equality of sets is introduced as an abbreviation:
\[x^{(k)}=_k y^{(k)}\leftrightharpoons \forall n(n\in_k x\equiv n\in_k y).\]

For brevity we will often omit indices in $=_k$ and $\in_k$.
We fix a standard numbering of pairs of natural numbers and denote $(m,n)$ the number of pair $m,n$ in this numbering.

Axiomatic theory $SA$ has the following axioms.

1. Classical predicate logic with equality.

1. Peano axioms.

$\quad \neg(n+1= 0).$

$\quad n+1=m+1\supset n=m.$

$\quad n+0=n.$

$\quad n+(m+1)=(n+m)+1.$

$\quad n\cdot0=0.$

$\quad n\cdot (m+1)=n\cdot m+n.$

3. Induction axiom.
$\quad \varphi(0)\wedge\forall n[\varphi(n)\supset\varphi(n+1)]\supset\forall n\varphi(n),$
\\where $\varphi$ is any formula of $SA$.

4. Comprehension axiom.
$\quad \exists z^{(k)}\forall n(n\in z\equiv\varphi(n)),$
\\where $\varphi$ is a $k$-simple formula not containing the variable $z^{(k)}$.

5. Choice axiom.
\[\quad \forall n\exists! x^{(k)}
\varphi(n,x)\supset \exists y^{(k+1)}\forall n\exists x^{(k)}[\varphi(n,x)\wedge \forall m(m\in x\equiv (n,m)\in y)],\]
where $\varphi$ is a $k$-simple formula.

This completes the definition of the theory $SA$. For $s\geqslant 0$ we denote $SA_s$ the fragment of $SA$ containing only sorts not greater than $s$. Thus, $SA_0$ is the Peano arithmetic $PA$.

In the rest of this section we compare $SA$ with some standard second-order arithmetics.

\subsection{Predicative second-order arithmetic $Ar$}

The language of theory $Ar$ has variables:
\begin{list}{}{}
\item $n_1, n_2, \ldots,m,n, \ldots$ over natural numbers and 
\item $x_1, x_2,\ldots, x,y,\ldots$ over sets of natural numbers.
\end{list}

The language of $Ar$ has two numerical constants 0 and 1, and functional symbols $\cdot$ and +. There are two predicate  symbols: = (equality of natural numbers) and $\in$.

Numerical terms are constructed from numerical variables and  constants using functional symbols.
Atomic formulas are:

$t=\tau$; $t\in x$, where $t$ and $\tau$ are numerical terms.

Formulas are constructed from atomic formulas and $\bot$ using logical connectives and quantifiers.

Axiomatic theory $Ar$ has the following axioms.

1. Classical predicate logic with equality.

1. Peano axioms (the same as in $SA$).

3. Induction axiom.
$\quad \varphi(0)\wedge\forall n[\varphi(n)\supset\varphi(n+1)]\supset\forall n\varphi(n),$
\\where $\varphi$ is any formula of $Ar$.

4. Comprehension axiom.
$\quad \exists z\forall n(n\in z\equiv\varphi(n)),$
where formula $\varphi$ has no quantifiers over set variables and does not contain the variable $z$.

This completes the definition of the theory $Ar$.

\subsection{Interpretations of $SA$ in weak second-order arithmetics}

Clearly, $SA$ is proof-theoretically stronger than $Ar$. However, $SA$ can be interpreted in some extensions of $Ar$.

Equality of sets is introduced in $Ar$ as an abbreviation:
\[x= y\leftrightharpoons \forall n(n\in x\equiv n\in y).\]

We define the following two formulas in the language of $Ar$.
\smallskip

$(AC!)\quad \forall k\exists! x\varphi(k,x)\supset\exists y\forall k\exists x[\varphi(k,x)\wedge \forall m(m\in x\equiv (k,m)\in y)]
,$
\smallskip
\\
where $\varphi$ has no quantifiers over set variables and does not contain the variable $y$.
\smallskip

$(\Delta^1_1-C)\quad \forall n[\forall v\varphi(n,v)\equiv
\exists u\psi(n,u)]
\supset\exists z\forall n[n\in z\equiv \exists u\psi(n,u)],$
\smallskip
\\where formulas $\varphi$ and $\psi$ have no quantifiers over set variables and do not contain the variable $z$.

\begin{theorem}
\begin{enumerate}

\item The theory $SA$ is interpretable in $Ar+(AC!)$.
\item The theory  $Ar+(AC!)$ is a sub-theory of $Ar+(\Delta^1_1-C)$. 
\end{enumerate}
\label{theorem:SA_int}
\end{theorem}

\begin{proof}
1. For a formula $\varphi$ of $SA$ we define its interpretation $\varphi^\wedge$ by induction on the complexity of $\varphi$.

$(t=\tau)^\wedge\leftrightharpoons t=\tau$.
\medskip

$\left(t\in_k x_i^{(k)}\right)^\wedge\leftrightharpoons t\in x_{(k,i)}$.
\medskip

$\bot^\wedge\leftrightharpoons\bot.$
\medskip

$(\psi\diamond \chi)^\wedge\leftrightharpoons \psi^\wedge\diamond \chi^\wedge.$
\medskip

$(Qn\psi)^\wedge\leftrightharpoons Qn\psi^\wedge.$
\medskip

$\left(Q x_i^{(k)}\psi\right)^\wedge\leftrightharpoons Qx_{(k,i)}\psi^\wedge.$
\medskip

Thus, $\varphi^\wedge$ is obtained from $\varphi$ by removing all sorts and renaming all set variables. Clearly, each axiom of $Ar+(AC!)$ can be obtained from a corresponding axiom of $SA$ in the same way. Therefore:
\[SA\vdash\varphi\quad\Rightarrow \quad Ar+(AC!)\vdash\varphi^\wedge\quad\Rightarrow\quad Ar+(AC!)\vdash\overline{\overline{\varphi}}^\wedge.\]

2. It is sufficient to show that $(AC!)$ is derived in $Ar+(\Delta_1^1-C)$. 

Consider a formula $\varphi$ that has no quantifiers over set variables and does not contain $y$. Denote:
\[\psi(n,x)\leftrightharpoons \exists k,m[n=(k,m)\wedge\varphi(k,x)\wedge m\in x],\]
\[\chi(n,x)\leftrightharpoons \exists k,m [n=(k,m)\wedge(\varphi(k,x)\supset m\in x)]
.\]

Assume the premise in $(AC!)$, that is 
\begin{equation}
\forall k\exists!x\varphi(k,x).
\label{eq:SA1}
\end{equation}

Then for any $n$:
\begin{equation}
\forall v\chi(n,v)\equiv\exists u\psi(n,u).
\label{eq:SA2}
\end{equation}

By axiom $(\Delta_1^1-C)$ there exists $y$ such that 
\begin{equation}
\forall n[n\in y\equiv\exists u\psi(n,u)].
\label{eq:SA3}
\end{equation}

It remains to prove: $\forall k\exists x[\varphi(k,x)\wedge\forall m(m\in x\equiv(k,m)\in y)].$ 

Consider an arbitrary $k$. By (\ref{eq:SA1}) there exists $x$ such that 
\begin{equation}
\varphi(k,x).
\label{eq:SA4}
\end{equation}

If $m\in x$, then for $n=(k,m)$ we have $\psi(n,x)$ and $n\in y$ by (\ref{eq:SA3}).

If $(k,m)\in y$, then for $n=(k,m)$ we have $\exists u\psi(n,u)$ by (\ref{eq:SA3}) and $\forall v\chi(n,v)$ by (\ref{eq:SA2}). So $\chi(n,x)$. By (\ref{eq:SA4}) $\varphi(k,x)$ and
$m\in x$ by the definition of $\chi$.
\end{proof}

\begin{corollary}
Theory $SA$ is interpretable in $Ar+(\Delta^1_1-C)$.
\end{corollary}

\section{Theory of arithmetical truth $PATr$} 
\label{sec:main}

This theory was introduced in \cite{kach14}. Theory $PATr$ is based on the axiomatic theory $PA$ for the first-order arithmetic. The language of $PATr$ is obtained from the language of $PA$ by adding predicate symbols $Tr_k (m,l), 
k=1,2,\ldots.$

For any $s\geqslant 1$, the language $PATr_s$ is obtained from the language of $PA$ by adding predicate symbols $Tr_k (m,l), 1\leqslant k\leqslant s.$ The language $PATr_0$ is just the language of $PA$.

Let us fix G\"{o}del numbering of expressions of the language $PATr$. It will be clear from context whether we use G\"{o}del numbering for expressions of $PATr$ or $BT$.
Next we introduce some arithmetical formulas.

$Form(k,m)\leftrightharpoons$ ``$m$ is the G\"{o}del number of a formula of $PATr_k$".

$Subform(m,r)\leftrightharpoons$ ``$r$ is the G\"{o}del number of a subformula of the formula with G\"{o}del number $m$".

$Param(m,i)\leftrightharpoons$ ``$n_i$ is a parameter of the expression of $PATr$ with G\"{o}del number $m$".

The following formula means that a sequence $l$ is an evaluation of all parameters of the expression with G\"{o}del number $m$:

$Ev(m,l)\leftrightharpoons (\forall i\leqslant m)[Param(m,i)\supset lh(l)\geqslant i].$ 

Clearly, the last four formulas define primitive recursive relations. 

We denote $eval$ and $subst$ the primitive recursive functions such that:

$eval(m,l)$ equals the value of term $t_m$ under evaluation $l$;

$subst(l,i,n)$ equals the evaluation $l$, in which the $i$-th  element is substituted by $n$.

Axiomatic theory $PATr$ has classical predicate logic with equality and the following non-logical axioms.

1. Peano axioms (the same as in $SA$).
\smallskip

2. Induction axiom.
$\quad \varphi(0)\wedge\forall n[\varphi(n)\supset\varphi(n+1)]\supset\forall n\varphi(n),$ where $\varphi$ is any formula of $PATr$.
\smallskip

3. Axioms for truth predicates (for any $k\geqslant 1$).

$\text{(Tr1)  }Tr_k(m,l)\supset Form(\overline{k-1},m)\wedge Ev(m,l);$

$\text{(Tr2)  }Ev(m,l)\wedge ``\varphi_m\textit{ is }t_i=t_j"
\supset [Tr_k(m,l)\equiv (eval(i,l)=eval(j,l))];$
\begin{multline*}
\text{(Tr3)  }
Ev(m,l)\wedge ``\varphi_m\textit{ is }Tr_k(t_i,t_j)"
\\
\supset [Tr_{k+1}(m,l)\equiv Tr_k(eval(i,l),eval(j,l))];
\end{multline*}

$\text{(Tr4)  }\neg Tr_k(\llcorner\bot\lrcorner,l)$;
\begin{multline*}
\text{(Tr5)  } Ev(m,l)\wedge  ``\varphi_m\textit{ is }\varphi_i\diamond\varphi_j"
\supset [Tr_k(m,l)\equiv (Tr_k(i,l)\diamond Tr_k(j,l))];
\end{multline*}
\begin{multline*}
\text{(Tr6)  } Ev(m,l)\wedge ``\varphi_m\textit{ is }Q n_i\varphi_j"
\supset [Tr_k(m,l)\equiv Q n Tr_k(j,subst(l,i,n))].
\end{multline*}

The axioms (Tr1)-(Tr6) describe $Tr_k$ as the truth predicate for formulas of $PATr_{k-1}$; that is, $Tr_k(m,l)$ means that the formula $\varphi_m$ is true under evaluation $l$.

This completes the definition of the theory $PATr$. Denote $PATr_s$ the fragment of $PATr$ in the language $PATr_s$. Clearly, $PATr_0$ is just the first-order arithmetic $PA$.

\section{Mutual interpretability of theories $BT$, $PATr$ and $SA$}
\label{section:interpretability}

\subsection{Interpretation of $BT$ in $PATr$}
In \cite{kach14} we constructed an interpretation $\varphi\rightarrow\varphi^\triangle$ and proved the following theorem and corollary.

\begin{theorem}
For $s\geqslant 0:$
if $BT_s\vdash \varphi$, then $PATr_s\vdash\overline{\overline{\varphi}}^\triangle.$
\label{theorem:BT-PATr}
\end{theorem}
\begin{corollary}
If $BT\vdash \varphi$, then $PATr\vdash\overline{\overline{\varphi}}^\triangle.$
\label{coro:BT-PATr}
\end{corollary}

\subsection{Interpretation of $PATr$ in $SA$}

For $k\geqslant 1$ denote $\widetilde{x}$ the list of variables $x_1^{(2)},\ldots,x_{k-1}^{(k)}$; when $k=1$, this list is empty. We define the following three formulas in $SA$.
\begin{multline*}
A_k(r,l,\widetilde{x},y^{(k)})
\leftrightharpoons \exists i,j\left\lbrace \left[``\varphi_r\textit{ is }t_i=t_j"\wedge eval(i,l)=eval(j,l)
\right] 
\right.
\\
\left.
\vee \bigvee_{q=1}^{k-1}\left[ ``\varphi_r\textit{ is }Tr_q(t_i,t_j)"\wedge (eval(i,l),(eval(i,l),eval(j,l)))\in x_q\right] 
\right.
\\
\left.
\vee \left[``\varphi_r\textit{ is }\varphi_i\diamond\varphi_j"\wedge ((i,l)\in y\diamond (j,l)\in y)\right]
\right.
\\
\left.
\vee \left[ ``\varphi_r\textit{ is }Q n_i\varphi_j"\wedge Q n \left[ (j,subst(l,i,n))\in y\right]\right]
\right\rbrace.
\end{multline*}
\begin{multline*}
FTrset_k(m,\widetilde{x},y^{(k)})\leftrightharpoons Form(\overline{k-1},m)\wedge \forall p\left\lbrace p\in y
\right.
\\
\left.
\equiv\exists r,l\left[p=(r,l)\wedge Subform(m,r)\wedge Ev(r,l)\wedge A_k(r,l,\widetilde{x},y) \right]\right\rbrace.
\end{multline*}

The last formula means that set $y$ contains G\"{o}del numbers of all true evaluated subformulas of formula $\varphi_m$ of $PATr_{k-1}$ given that $x_1,\ldots,x_{k-1}$ are corresponding truth sets for formulas of $PATr_0,\ldots,PATr_{k-2}$, respectively.
\begin{multline*}
Trset_k(\widetilde{x},z^{(k+1)})\leftrightharpoons
(\forall p\in z)\exists m,q[p=(m,q)
\wedge Form(\overline{k-1},m)]
\\
\wedge\forall m\left\lbrace Form(\overline{k-1},m)
\supset \exists y^{(k)}\left[FTrset_k(m,\widetilde{x},y)\wedge\forall n(n\in y\equiv(m,n)\in z)\right]\right\rbrace.
\end{multline*}

The last formula means that set $z$ contains G\"{o}del numbers of all true evaluated formulas of $PATr_{k-1}$ assuming that $x_1,\ldots,x_{k-1}$ are corresponding truth sets for formulas of $PATr_0,\ldots,PATr_{k-2}$, respectively.

By the comprehension axiom and the definition of set equality we have: 
\[SA\vdash\exists!z^{(k)}\forall n(n\in z\equiv\varphi(n)).\]

So for a $k$-simple formula $\varphi$ we can introduce in $SA$ a functional symbol $\{n\mid \varphi(n)\}$ of sort $k$.

We can introduce the following notations:
\[[x]_m=\{n\in x\mid\exists r,l[n=(r,l)\wedge SubForm(m,r)]\};\]
\[\emptyset^{(k)}=\{n\mid \neg(n=n)\}.\]

\begin{lemma}
The following formulas are derived in $SA_s$.
\medskip

$1. \; Form(\overline{k-1},m)
\supset\exists! y^{(k)} 
FTrset_k(m,\widetilde{x},y),\textit{ where }1\leqslant k\leqslant s.$
\begin{multline*}
2. \; Form(\overline{k-1},m)\wedge  SubForm(m,r)\wedge FTrset_k(m,\widetilde{x},y)
\\
\supset FTrset_k(r,\widetilde{x},[y]_r),\textit{ where }1\leqslant k\leqslant s.
\end{multline*} 

$3. \;\exists! z^{(k+1)} 
Trset_k(\widetilde{x},z),\textit{ where }1\leqslant k< s.$
\label{lemma:PATr-SA-1}
\end{lemma}
\begin{proof}
1. Proof is by induction on $m$ using the definition of $FTrset_k$.

2. This follows from part 1 and the definition of $FTrset_k$.

3. Fix $\widetilde{x}$. Denote $\psi(m,y^{(k)})\leftrightharpoons
\medskip
\\ \left[Form(\overline{k-1},m)\supset  FTrset_k(m,\widetilde{x},y)\right]\wedge \left[\neg Form(\overline{k-1},m)\supset y=\emptyset^{(k)}\right].$
\medskip

By part 1, $\forall m\exists!y^{(k)}\psi(m,y)$. 
By the choice axiom there exists $v^{(k+1)}$ such that:
\[\forall m\exists y^{(k)}[\psi(m,y)\wedge\forall n(n\in y\equiv(m,n)\in v)].\]

Then for $z^{(k+1)}=\{q\in v\mid\exists m,n(q=(m,n))\}$ we have $Trset_k(\widetilde{x},z)$. 

The uniqueness follows from part 1 and the definition of $Trset_k$. 
\end{proof}

By Lemma \ref{lemma:PATr-SA-1}.3 for $k=1$, $SA_2\vdash\exists!z^{(2)}Trset_1(z)$, so we can introduce in $SA_2$ a constant $a_1^{(2)}$ such that $SA_2\vdash Trset_1(a_1)$. 

By Lemma \ref{lemma:PATr-SA-1}.3 for $k=2$, $SA_3\vdash\exists!z^{(3)}Trset_2(a_1,z)$, so we can introduce in $SA_3$ a constant $a_2^{(3)}$ such that $SA_3\vdash Trset_2(a_1, a_2)$. 

Continuing by induction, we can introduce in $SA_{k+1}$ a constant $a_k^{(k+1)}$ such that 
\begin{equation}
SA_{k+1}\vdash Trset_k(a_1,\ldots, a_{k-1}, a_k).
\label{eq:PATr-SA-1}
\end{equation}

By Lemma \ref{lemma:PATr-SA-1}.1,
\begin{multline*}
SA_{k}\vdash \forall m\exists! y^{(k)} \left\lbrace 
\left[Form(\overline{k-1},m)\supset FTrset_k(m,a_1,\ldots, a_{k-1},y)\right]
\right.
\\
\left.
\wedge
\left[\neg Form(\overline{k-1},m)\supset y=\emptyset^{(k)}\right]
\right\rbrace .
\end{multline*} 

So we can introduce in $SA_{k}$ a functional symbol $g_k(m)^{(k)}$ such that 
\begin{equation}
SA_{k}\vdash Form(\overline{k-1},m)\supset FTrset_k(m,a_1,\ldots, a_{k-1},g_k(m)).
\label{eq:PATr-SA-2}
\end{equation}

\begin{lemma}
For $1\leqslant k< s:$
\[SA_s\vdash n\in_k g_k(m)\equiv (m,n)\in_{k+1} a_k.\]
\label{lemma:PATr-SA-2}
\end{lemma}
\begin{proof}
Proof follows from the definitions and formulas (\ref{eq:PATr-SA-1}), (\ref{eq:PATr-SA-2}).
\end{proof}

Next for each formula $\varphi$ of $PATr$ we define its interpretation $\varphi^\sim$ in $SA$ by induction on the complexity of $\varphi$.

$(t=\tau)^\sim\leftrightharpoons t=\tau$.
\medskip

$Tr_k(t,\tau)^\sim\leftrightharpoons (t,\tau)\in_k g_k(t)$, $k\geqslant 1$.
\medskip

$\bot^\sim\leftrightharpoons\bot.$
\medskip

$(\psi\diamond \chi)^\sim\leftrightharpoons \psi^\sim\diamond \chi^\sim.$
\medskip

$(Qn\psi)^\sim\leftrightharpoons Qn\psi^\sim.$
\medskip

Clearly, if $\varphi$ is a formula of $PATr_s$, then $\varphi^\sim$ is a formula of $SA_s$ $(s\geqslant 0)$.

\begin{theorem}
\begin{enumerate}
\item For an arithmetical formula $\varphi$, $\varphi^\sim$ is the same as $\varphi$.
\item For $s\geqslant 0:$
if $PATr_s\vdash \varphi$, then $SA_s\vdash(\overline{\overline{\varphi}})^\sim.$
\end{enumerate}
\label{theorem:PATr-SA}
\end{theorem}
\begin{proof}
1. This follows immediately from the definition of $\varphi^\sim$.

2. Both $PATr_0$ and $SA_0$ are the same as the first-order arithmetic $PA$. 

For $s\geqslant 1$ proof is by induction on the length of derivation of $\varphi$. Since logical connectives and quantifiers are preserved in this interpretation, the statement is obvious for the induction axiom and the classical predicate logic. Peano axioms are the same in both theories.

For axioms $(Tr1)-(Tr6)$ the statement follows from Lemmas \ref{lemma:PATr-SA-1}, \ref{lemma:PATr-SA-2} and the definitions of $a_k$ and $g_k$.
\end{proof}

\begin{corollary}
If $PATr\vdash \varphi$, then $SA\vdash(\overline{\overline{\varphi}})^\sim.$
\label{coro:PATr-SA}
\end{corollary}

\subsection{Interpretation of $SA$ in $BT$ with classical logic}

We will use the following notations in $BT$.

$\{n\}^0=n$, $\{n\}^{k+1}=\{\{n\}^k\}.$
\medskip

For $k\geqslant 1:$ $M_k(X^k)\leftrightharpoons \left(\forall Z^{k-1}\in X\right)\exists n\left( \{n\}^{k-1}\simeq Z\right).$
$M_k(X^k)$ means that $X$ is an interpretation of a set of natural numbers of sort $k$.

 For every formula $\varphi$ of $SA$ we define its interpretation $\varphi^*$ by induction on the complexity of $\varphi$.

$(t=\tau)^*\leftrightharpoons t=\tau$ (we identify arithmetical terms in $SA$ with corresponding arithmetical terms in $BT$).
\medskip

$\left(t\in_k x_i^{(k)}\right)^*\leftrightharpoons \{t\}^{k-1}\in_{k-1}X^k_i$.
\medskip

$\bot^*\leftrightharpoons\bot.$
\medskip

$(\psi\diamond \chi)^*\leftrightharpoons \psi^*\diamond \chi^*.$
\medskip

$(Qn\psi)^*\leftrightharpoons Qn\psi^*.$
\medskip

$\left(\forall x_i^{(k)}\psi\right) ^*\leftrightharpoons \forall X_i^k\left[M_k(X_i)\supset\psi^*\right] .$
\medskip

$\left(\exists x_i^{(k)}\psi\right) ^*\leftrightharpoons \exists X_i^k\left[M_k(X_i)\wedge\psi^*\right] .$
\medskip

Clearly, if $\varphi$ is a formula of $SA_s$, then $\varphi^*$ is a formula of $BT_s$ $(s\geqslant 0)$.

The following is proven by induction on $k$:
\begin{equation}
\left(\{n\}^k\simeq x\right) \textit{ is a 0-elementary formula.}
\label{eq:SA_1}
\end{equation} 

This implies:
\begin{equation}
\left(t\in_k x_i^{(k)}\right)^* \textit{ is a }k\textit{-elementary formula;}
\label{eq:SA_2}
\end{equation} 
\begin{equation}
M_k(X^k) \textit{ is a }k\textit{-elementary formula.}
\label{eq:SA_3}
\end{equation} 

Using (\ref{eq:SA_1}) and (\ref{eq:SA_2}) the following is proven by induction on the complexity of $\varphi$:
\begin{equation}
\textit{If }\varphi\textit{ is a }k\textit{-simple formula, then }\varphi^*\textit{is a }k\textit{-elementary formula.}
\label{eq:SA_4}
\end{equation}

We denote $BT^{cl}$ the theory $BT$ with classical logic and $BT_s^{cl}$ the theory $BT_s$ with classical logic.

\begin{theorem}
\begin{enumerate}
\item If $\psi$ is an arithmetical formula (that is, a formula of $PA$), then $BT_0^{cl}\vdash\psi^*\equiv\psi.$
\item For $s\geqslant 0:$ if $SA_s\vdash \psi$, then $BT_s^{cl}\vdash(\overline{\overline{\psi}})^*$.
\end{enumerate}
\label{theorem:SA-BT}
\end{theorem}

\begin{proof}
1. This follows from the definition of $\psi^*$.

2. For $s=0$ it is obvious.

For $s\geqslant 1$ proof is by induction on the length of derivation of $\psi$.
We will consider only the comprehension and choice axioms, for others proof is quite straightforward. 

Suppose $\psi$ is the comprehension axiom:
\[\exists x^{(k)}\forall n(n\in x\equiv\varphi(n)),\]
where $1\leqslant k\leqslant s$ and $\varphi$ is a $k$-simple formula not containing $x^{(k)}$.

To prove $\overline{\overline{\psi}}^*$, it is sufficient to show:
\begin{equation}
\exists X^k\left[M_k(X)\wedge\forall n\left( \{n\}^{k-1}\in X\equiv\varphi(n)^*\right) \right].
\label{eq:SA-BT}
\end{equation}

Denote $\chi(Z^{k-1})\leftrightharpoons\exists n,y\left(\{n\}^{k-1}\simeq y\wedge y=_{0,k-1}Z\wedge\varphi(n)^* \right).$ 

By (\ref{eq:SA_1}) and (\ref{eq:SA_4}), $\chi$ is a $k$-elementary formula and by the comprehension axiom in $BT_s$ there exists $X^k$ such that $\forall Z^{k-1}(Z\in X\equiv\chi(Z))$. For this $X$ we have:
\[M_k(X)\wedge\forall n\left[\{n\}^{k-1}\in X\equiv \varphi(n)^*\right],\]
which proves (\ref{eq:SA-BT}). 

Suppose $\psi$ is the choice axiom:
\begin{multline*}
\forall n\exists x^{(k)}
\left[ \varphi(n,x)\wedge\forall z^{(k)}
\left(\varphi(n,z)\supset\forall  m(m\in x\equiv m\in z)\right)\right] 
\\ 
\supset \exists y^{(k+1)}\forall n\exists x^{(k)}\left[ \varphi(n,x)\wedge \forall m(m\in x\equiv (n,m)\in y)\right],
\end{multline*}
where $1\leqslant k< s$ and $\varphi$ is a $k$-simple formula, so it does not contain the variable $y^{(k+1)}$.

To prove $\overline{\overline{\psi}}^*$, it is sufficient to show:
\begin{multline}
\forall n\exists X^k\left[M_k(X)\wedge \varphi(n,X)^*\wedge\forall Z^k \left[M_k(Z)
\wedge\varphi(n,Z)^*
\right.
\right.
\\
\left.
\left.
\supset \forall m\left(\{m\}^{k-1}\in X\equiv\{m\}^{k-1}\in Z\right)\right] \right]
\supset\exists Y^{k+1}
\left[M_{k+1}(Y)
\right.
\\
\left.
\wedge\forall n\exists X^k\left(M_k(X)\wedge\varphi(n,X)^*
\wedge\forall m\left(\{m\}^{k-1}\in X\equiv\{(n,m)\}^{k}\in Y\right)
 \right) \right].
\label{eq:SA-BT-1}
\end{multline}

Denote 
\[\xi(U^k)\leftrightharpoons\exists n,m,X^k\left[M_k(X)\wedge\varphi(n,X)^*\wedge\{(n,m)\}^{k}\simeq U\wedge\{m\}^{k-1}\in X \right].\]

It follows from (\ref{eq:SA_1})-(\ref{eq:SA_3}) that $\xi$ is a $(k+1)$-elementary formula. 

Suppose the premise of (\ref{eq:SA-BT-1}). By the comprehension axiom in $BT_s^{cl}$ there exists $Y^{k+1}$ such that 
\[\forall U^k(U\in Y\equiv\xi(U)).\]
Then $M_{k+1}(Y)$. Due to the premise of (\ref{eq:SA-BT-1}), for any $n$ there is $X^k$ such that $M_k(X)\wedge\varphi(n,X)^*$ and $\forall m\left(\{m\}^{k-1}\in X\equiv\{(n,m)\}^{k}\in Y\right).$ 
\end{proof}

\begin{corollary}
If $SA\vdash \psi$, then $BT^{cl}\vdash(\overline{\overline{\psi}})^*$.
\label{coro:SA-BT}
\end{corollary}

\subsection{Interpretation of $BT^{cl}$ in $BT$}

We generalise the negative interpretation defined in \cite{frie73} to the theory $BT$. 
For any constant $a$, which is not a comprehension constant $\underline{c}_n$, we define $a^-=a$. 

By simultaneous induction on G\"{o}del numbers we define interpretation $\varphi^-$ for any formula $\varphi$ of $BT$ and interpretation $\underline{c}_n^-$ for any comprehension constant $\underline{c}_n$.

1) Interpretation $\varphi^-$.

If $\varphi$ is an atomic formula, then $\varphi^-\leftrightharpoons \neg\neg\tilde{\varphi}$, where $\tilde{\varphi}$ is obtained from $\varphi$ by replacing each constant $a$ by $a^-$.

$\bot^-\leftrightharpoons\bot.$
\medskip

$(\psi\diamond \chi)^-\leftrightharpoons \psi^-\diamond \chi^-$ if $\diamond$ is a connective $\wedge$ or $\supset$.
\medskip

$(\psi\vee \chi)^-\leftrightharpoons \neg\neg(\psi^-\vee \chi^-).$
\medskip

$\left(\forall X\psi\right)^-\leftrightharpoons \forall X\psi^-,$
\medskip

$\left(\exists X\psi\right)^-\leftrightharpoons \neg\neg\exists X\psi^-,$ where $X$ is any variable of $BT$.
\medskip

2) In case when $n$ has the form $\llcorner Z^k.\widetilde{X}.\varphi\lrcorner$, we define $\underline{c}_n^-=\underline{c}_{n^-}$, where $n^-=\llcorner Z^k.\widetilde{X}.\varphi^-\lrcorner$.  In other cases we take $\underline{c}_n^-=0$.

For any external term $t$ its interpretation $t^-$ is obtained from $t$ by replacing each constant $a$ by its interpretation $a^-$.

\begin{lemma}
\begin{enumerate}
\item For any external terms $t$ and $\tau$ with no types $>s$:
\[BT_s\vdash(t\simeq\tau)^-\equiv\neg\neg(t^-\simeq\tau^-).\]
\item For any formula $\varphi$ of $BT_s$:
\[BT_s\vdash \neg\neg\varphi^-\equiv\varphi^-.\]
\item If an arithmetical formula $\varphi$ expresses a primitive recursive predicate, then:
\[BT_0\vdash \varphi^-\equiv\varphi.\]
\end{enumerate}
\label{lemma:BTcl-BT}
\end{lemma}
\begin{proof}
1. Proof is by induction on the construction of the terms.

2. Proof is by induction on the complexity of $\varphi$.

3. Since $\varphi$ is an arithmetical formula, then $\varphi^-\equiv\neg\neg\varphi$ (it is proven by induction on the complexity of $\varphi$). Since $\varphi$ expresses a primitive recursive predicate, then $\neg\neg\varphi\equiv\varphi$.
\end{proof}

\begin{theorem}
For $s\geqslant 0:$ if $BT^{cl}_s\vdash \psi$, then $BT_s\vdash\psi^-$.
\label{theorem:BTcl-BT}
\end{theorem}

\begin{proof}
Proof is by induction on the length of derivation of $\psi$. For axioms and derivation rules of classical logic the proof is standard, see, for example, \cite{drag87}. We consider only the case when $\psi$ is the comprehension axiom:
\[\exists U^{k+1}\left[\;\underline{c}_n(\widetilde{X})\simeq U\wedge\forall Z^k(Z\in U\equiv \varphi)
 \right],\]
where $0\leqslant k< s$, 
$n=\llcorner Z^k.\widetilde{X}.\varphi\lrcorner$, and $\varphi$ is a $(k+1)$-elementary formula. By the definition $\varphi^-$ is also a $(k+1)$-elementary formula, so 
\[\exists U^{k+1}\left[\;\underline{c}_{n^-}(\widetilde{X})\simeq U\wedge\forall Z^k(Z\in U\equiv \varphi^-) \right].\]

By Lemma \ref{lemma:BTcl-BT}.2, $\neg\neg\varphi^-\equiv\varphi^-$. Therefore:
\[\neg\neg\exists U^{k+1}\left[\;\neg\neg(\underline{c}_{n^-}(\widetilde{X})\simeq U)\wedge\forall Z^k(\neg\neg(Z\in U)\equiv \varphi^-) \right].\]

By Lemma \ref{lemma:BTcl-BT}.1, $\neg\neg(\underline{c}_{n^-}(\widetilde{X})\simeq U)\equiv(\underline{c}_{n}(\widetilde{X})\simeq U)^-.$ Therefore
\[\neg\neg\exists U^{k+1}\left[\;(\underline{c}_{n}(\widetilde{X})\simeq U)^-\wedge\forall Z^k(Z\in U\equiv \varphi)^- \right],\]
which is $\psi^-$.
\end{proof}

\begin{corollary}
If $BT^{cl}\vdash \psi$, then $BT\vdash\psi^-.$
\label{coro:BTcl-BT}
\end{corollary}

\subsection{Summary of interpretabilities}
\begin{theorem}
The fragments $BT_s,BT_s^{cl},PATr_s$ and $SA_s$ are interpretable in one another.
\label{theorem:summary1}
\end{theorem}
\begin{proof}
Using symbol $\rightarrow$ for ``interpretable" we can summarise the results of this section:

$BT_s\rightarrow PATr_s$ (Theorem \ref{theorem:BT-PATr});
\medskip

$PATr_s\rightarrow SA_s$ (Theorem \ref{theorem:PATr-SA}.2);
\medskip

$SA_s\rightarrow BT^{cl}_s$ (Theorem \ref{theorem:SA-BT}.2);
\medskip

$BT^{cl}_s\rightarrow BT_s$ (Theorem \ref{theorem:BTcl-BT}).
\medskip

This means that the four fragments are interpretable in one another.
\end{proof}

\begin{theorem}
The theories $BT,BT^{cl},PATr$ and $SA$ are interpretable in one another.
\label{theorem:summary2}
\end{theorem}
\begin{proof}
This follows from Corollaries \ref{coro:BT-PATr}, \ref{coro:PATr-SA}, \ref{coro:SA-BT} and \ref{coro:BTcl-BT} similarly to the previous theorem.
\end{proof}

\subsection{Comparison of $SA$ with Simpson's subsystems of second order arithmetic}

In \cite{simp10} Simpson introduced several formal theories for reverse mathematics. All these theories are subsystems of second order arithmetic and have classical logic; most of the theories have only restricted induction axiom:
\[0\in X\wedge\forall n(n\in X\supset n+1\in X)\supset\forall n(n\in X).\]

Let us denote $SA^r$ the theory $SA$ where the induction axiom is restricted to formulas with no set quantifiers. Then $SA^r$ is equivalent to the theory $SA$ where the induction axiom has the form:
\[0\in z^{(k)}\wedge\forall n(n\in z\supset n+1\in z)\supset\forall n(n\in z),\; k\geqslant 1.\]

If we similarly restrict the induction axiom in the theories $BT$, $BT^{cl}$ and $PATr$, then the theorem about their mutual interpretability still holds, as well as the theorem about the mutual interpretability of their corresponding fragments.

With respect to proof-theoretical strength, the theory $SA^r$ is between the Simpson's theories $ACA_0$ (the second-order arithmetic with arithmetical comprehension) and $\triangle^1_1-CA_0$ (the second order arithmetic with $\triangle^1_1$ comprehension).

Ordinary mathematics can be developed in $SA$ in a similar way that Simpson \cite{simp10} develops it in the theory $ACA_0$. We believe that some definitions can be simplified in $SA$ due to its multi-sorted language but this requires more research. 

\section{Comparison of the proof-theoretical strengths of  fragments $BT_s$, $PATr_s$ and $SA_s$}
\label{section:comparison}

It follows from Theorem \ref{theorem:summary2} that the theories $BT,BT^{cl}$, $PATr$ and $SA$ are equiconsistent. It follows from Theorem \ref{theorem:summary1} that for each $s\geqslant 0$ the fragments $BT_s,BT_s^{cl},PATr_s$ and $SA_s$ are equiconsistent. 

\begin{theorem}
For $s\geqslant 0:$
\begin{enumerate}
\item $PATr_{s+1}\vdash Con_{PATr_s};$
\item $BT_{s+1}\vdash Con_{BT_s};$
\item $SA_{s+1}\vdash Con_{SA_s}.$
\end{enumerate}
\label{theorem:strength}
\end{theorem}

\begin{proof}
1. This was proven in \cite{kach14} as Lemma 3.b.

2. By Theorem 2 in \cite{kach14}, 
\begin{equation}
PATr_{s+1}\vdash Con_{BT_s}. 
\label{eq:Lobach}
\end{equation}

By Theorem \ref{theorem:PATr-SA}.1, 2, $SA_{s+1}\vdash Con_{BT_s}$, since $Con_{BT_s}$ is a closed arithmetical formula.

Similarly, by Theorem \ref{theorem:SA-BT}.1, 2,
$BT_{s+1}^{cl}\vdash Con_{BT_s}$ and by Theorem \ref{theorem:BTcl-BT}, $BT_{s+1}\vdash (Con_{BT_s})^-.$ 

The formula $Con_{BT_s}$ is $\forall n\neg Proof_{BT_s}(\llcorner \bot\lrcorner,n)$, which is a closed formula and $Proof_{BT_s}(\llcorner \bot\lrcorner,n)$ expresses a primitive recursive predicate. So by Lemma \ref{lemma:BTcl-BT}.3, 
\[(Con_{BT_s})^-\equiv\forall n\neg (Proof_{BT_s}(\llcorner \bot\lrcorner,n))^-\equiv\forall n\neg Proof_{BT_s}(\llcorner \bot\lrcorner,n)\equiv Con_{BT_s}.\]

Therefore $BT_{s+1}\vdash Con_{BT_s}$.

3. By formalising the proofs of Theorems \ref{theorem:SA-BT}.2 and \ref{theorem:BTcl-BT} in $PATR_{s+1}$ we get:
\[PATR_{s+1}\vdash Pv_{SA_s}(m)\supset Pv_{BT_s}(\llcorner(\overline{\overline{\varphi}}_m)^{*-}\lrcorner).\]

In particular, for $m=\llcorner\bot\lrcorner$ we have: 
\[PATR_{s+1}\vdash Pv_{SA_s}(\llcorner\bot\lrcorner)\supset Pv_{BT_s}(\llcorner\bot\lrcorner).\]

So $PATR_{s+1}\vdash\neg Pv_{BT_s}(\llcorner\bot\lrcorner)
\supset\neg Pv_{SA_s}(\llcorner\bot\lrcorner),$ that is 
\[PATR_{s+1}\vdash Con_{BT_s}\supset Con_{SA_s}\]
and by \eqref{eq:Lobach}, $PATR_{s+1}\vdash Con_{SA_s}.$

By Theorem \ref{theorem:PATr-SA}.1, 2, we get $SA_{s+1}\vdash Con_{SA_s}.$
\end{proof}

Clearly, the Beeson's theory $BEM+(CA)$ is the same as the fragment $BT_1$. This implies the next corollary.
\begin{corollary}
\begin{enumerate}
\item $BT_2\vdash Con_{BEM+(CA)}$.
\item  $BT\vdash Con_{BEM+(CA)}$.
\end{enumerate}
\end{corollary}

\section{Discussion}
\label{discussion}

In this paper we described the axiomatic theory $BT$, which is a suitable formal theory for developing constructive mathematics, due to its constructive properties such as the existence and disjunction properties, and consistency with the formal Church thesis. Also $BT$ is interpretable in relatively weak versions $Ar+(AC!)$ and $Ar+(\Delta_1^1-C)$ of second-order arithmetic.

We studied the proof-theoretical strength of $BT$ by comparing it with the axiomatic theories $PATr$ and $SA$, and we showed that all three theories are interpretable in one another. We also showed that the fragments $BT_s$, $PATr_s$ and $SA_s$ are interpretable in one another and that each of them is weaker than a corresponding next fragment. In particular, this means that each of the theories $BT$, $PATr$ and $SA$ is stronger than the predicative second-order arithmetic. 

Next we plan to use the advantage of multi-typed language of $BT$ to state and prove constructive versions of theorems of classical mathematics in $BT$. We also plan to investigate the consistency of $BT$ with a stronger version of the formal Church thesis.

\bibliography{Farida}

\end{document}